\documentclass[a4paper,12pt]{article}
\usepackage{amsmath,amsthm,amssymb,amsfonts}
\usepackage{bbm,bm}
\usepackage{latexsym}
\usepackage{mathrsfs}
\usepackage{threeparttable}
\usepackage{tabularx}
\usepackage{booktabs}
\usepackage{graphicx,subfigure}
\usepackage{color}
\usepackage{indentfirst}
\usepackage{geometry}
\usepackage{caption}
\usepackage{lineno}
\usepackage{abstract}
\usepackage{enumerate,mdwlist}
\usepackage[numbers,sort&compress]{natbib}
\usepackage[colorlinks,
            linkcolor=blue, 
            citecolor=red  
            ]{hyperref}
\usepackage{epstopdf}

\def \[{\begin{equation}}
\def \]{\end{equation}}

\newtheorem{thm}{Theorem}[section]
\newtheorem{prop}[thm]{Proposition}

\newtheorem{lem}[thm]{Lemma}
\newtheorem{cor}[thm]{Corollary}

\newtheorem{conj}[thm]{Conjecture}

\begin{document}

\setlength{\baselineskip}{20pt}
\begin{center}{\Large \bf Minimally \emph{k}-factor-critical graphs for some large \emph{k}}\footnote{This work
is supported by NSFC\,(Grant No. 11871256).}

\vspace{4mm}

{Jing Guo, Heping Zhang \footnote{The corresponding author.}
\renewcommand\thefootnote{}\footnote{E-mail addresses: guoj20@lzu.edu.cn (J. Guo), zhanghp@lzu.edu.cn (H. Zhang).}}

\vspace{2mm}

\footnotesize{ School of Mathematics and Statistics, Lanzhou University, Lanzhou, Gansu 730000, P. R. China}

\end{center}
\noindent {\bf Abstract}: A graph $G$ of order $n$ is said to be $k$-factor-critical
 for integers $1\leq k < n$, if the removal of any $k$
 vertices results in a graph with a perfect matching.
 $1$- and $2$-factor-critical graphs are the well-known
 factor-critical and bicritical graphs, respectively. A $k$-factor-critical graph $G$
 is called minimal if for any edge $e\in E(G)$, $G-e$ is not $k$-factor-critical.
 In 1998, O. Favaron and M. Shi conjectured that every minimally $k$-factor-critical graph
 of order $n$ has the minimum degree $k+1$ and confirmed it for $k=1, n-2, n-4$ and $n-6$.
 In this paper, we use a simple method to reprove the above result.
 As a main result, the further use of this method enables ones to
 prove the conjecture to be true for $k=n-8$.
 We also obtain that every minimally $(n-6)$-factor-critical graph of order $n$ has
 at most $n-\Delta(G)$ vertices with the maximum degree $\Delta(G)$ for $n-4\leq \Delta(G)\leq n-1$.

\vspace{2mm} \noindent{\bf Keywords}: Perfect matching;
Minimally $k$-factor-critical graph; Minimum degree.
\vspace{2mm}


 {\setcounter{section}{0}
\section{Introduction}\setcounter{equation}{0}

All graphs considered in this paper are finite,  undirected and simple.
Let $G$ be a graph with vertex set $V(G)$ and edge set  $E(G)$.
The {\em order} of  $G$ is the cardinality of $V(G)$.
For a vertex $x$ of G, let $d_G(x)$ be the degree of $x$ in $G$, i.e. the number of edges of $G$ incident with $x$, and let $\delta(G)$
and  $\Delta(G)$ denote the minimum degree and maximum degree of $G$ respectively.

A {\em matching} of $G$ is an edge subset of $G$ in which no two edges have a common end-vertex.
A matching $M$ of $G$ is said to be a {\em perfect matching} or a {\em 1-factor} if it covers all vertices of $G$.
A graph $G$ is called {\em factor-critical} if the removal of each vertex
of $G$ results in a graph with a perfect matching.  A graph with an edge is called {\em bicritical} if the removal
of each pair of distinct vertices of $G$ results in a graph with a perfect matching. A 3-connected bicritical graph is the so-called  {\em brick}.
Factor-critical and bricks were introduced by
T. Gallai \cite{GT} and L. Lov\'{a}sz \cite{LL}, respectively, which play important roles in Gallai-Edmonds Structure Theorem and in determining the dimensions of perfect matching polytopes and matching lattices; see a detailed monograph due to L. Lov\'{a}sz and M. D. Plummer \cite{LP}.

Generally, O. Favaron \cite{F} and Q. Yu \cite{Y} independently defined $k$-factor-critical graphs for any positive integer $k$.
A graph $G$ of order $n$ is said to be $k$-{\em factor-critical}
for positive integer $k<n$, if the removal of any $k$
vertices of $G$ results in a graph with a perfect matching. They   characterized  $k$-factor-critical graphs in Tutte's type and showed that such graphs are $(k+1)$-edge-connected.  To date there have been  many studies on $k$-factor-critical graphs;
see articles \cite{PS, L, PMD, N, LD, LDYQ, FFR, ZWZ} and a monograph \cite{YL}.

A graph $G$ is called {\em minimally $k$-factor-critical} if $G$ is $k$-factor-critical
but $G-e$ is not $k$-factor-critical for any $e\in E(G)$.
O. Favaron and M. Shi \cite{FS} studied some properties of minimally
$k$-factor-critical graphs and obtained an upper bound of minimum degree of
minimally $k$-factor-critical graphs as follows.

\begin{thm}[\cite{FS}] \label{degree} For a minimally $k$-factor-critical graph $G$ of
order $n\geq k+4$, $\delta(G)\leq \frac{n+k}{2}-1$.
If moreover $n\geq k+6$, then $\delta(G)\leq \frac{n+k}{2}-2$.
\end{thm}

From Theorem \ref{degree}, the following result is immediate.

\begin{cor}[\cite{FS}]\label{Special}
Let $G$ be a minimally $k$-factor-critical graph of order $n$.
If $k=n-2, n-4$ or $n-6$, then $\delta(G)=k+1$.
\end{cor}

O. Favaron and M. Shi  \cite{FS} also pointed out that  from the ear decomposition of factor-critical graphs (see \cite{LP}),
obviously a minimally 1-factor-critical graph has the minimum degree two. Further, since a minimally $k$-factor-critical graph is $(k+1)$-edge-connected and thus has the minimum degree at least $k+1$, O. Favaron and M. Shi asked a problem: does Corollary \ref{Special} hold  for general $k$?

Similarly, for minimally $q$-extendable graphs,
D. Lou and Q. Yu \cite{LY} conjectured that
any minimally $q$-extendable graph $G$ on $n$ vertices with $n\leq 4q$
has minimum degree $q + 1, 2q$ or $2q + 1$.
Afterward, Zhang et al. \cite{ZWL} formally reproposed the following conjecture
and pointed out that except the case $n=4q$, the conjecture of minimum degree of
minimally $q$-extendable graph is actually part of Conjecture \ref{conj}.

\begin{conj}[\cite{FS,ZWL}]\label{conj}
Let $G$ be a minimally $k$-factor-critical graph of order $n$  with $0\leq k<n$.
Then $\delta(G)=k+1$.
\end{conj}

From the above discussions we know that Conjecture \ref{conj} is true for $k=1, n-2, n-4, n-6$. To date Conjecture \ref{conj} remains open for $2\leq k\leq n-8$ of the same parity as $n$.

However, recently some great progresses have been made on related bicritical graphs. A brick is {\em minimal} if the removal of any edge results in a graph that is non-brick. From the construction of a brick
M. H. de Carvalho et al. \cite{CLM} proved that
every minimal  brick contains a vertex of degree three.
S. Norine and R. Thomas \cite{NT} proved that every minimal brick has at least
three vertices of degree three.
Latter, F. Lin et al. \cite{LZL} obtained that every minimal brick  has
at least four vertices of degree three.
At the same time, H. Bruhn and M. Stein \cite{BS} showed that
every minimal brick $G$ has at least $\frac{1}{9}|V(G)|$
vertices of degree at most four.


In this paper, we use a novel and simple method to reprove Corollary \ref{Special}.
Continuing this method, we can prove  Conjecture \ref{conj} to be true for $k=n-8$.
On the other hand, the only $(n-2)$-factor-critical graph of order $n$ is complete graph $K_{n}$. O. Favaron and M. Shi  \cite{FS} characterized minimally $(n-4)$-factor-critical graphs of order $n$ in the degree distribution. Finally we obtain that every minimally $(n-6)$-factor-critical graph of order $n$ has at most $n-\Delta(G)$ vertices with the maximum degree $\Delta(G)$
for $n-4\leq \Delta(G)\leq n-1$.

\section{Some preliminaries}

In this section we give some graph-theoretical terminology and notation, and some preliminary results for late use.
For a vertex $x$ of a graph $G$, the {\em neighborhood} $N(x)$ of $x$ is the set of vertices of $G$ adjacent to $x$,
and the {\em closed neighborhood} is $N[x]=N(x)\cup \{x\}$. Then
$\overline{N[x]}:=V(G)\setminus N[x]$ is called the {\em non-neighborhood} of $x$ in $G$, which has a critical role in subsequent discussions.

A vertex of a graph $G$ with degree one is called a {\em pendent vertex}.
An {\em independent set} in a graph is a set of pairwise nonadjacent vertices.
For a set $S\subseteq V(G)$, let $G[S]$ denote  the subgraph of $G$ induced by $S$ in $G$,
and $G-S=G[V(G)-S]$.
For an edge $e$ of $G$, $G-e$ stands for the graph with vertex set
$V(G)$ and edge set $E(G)-\{e\}$. Similarly, for distinct vertices $u$ and $v$ with $e=uv \notin E(G)$, $G+e$ stands for the
graph with vertex set $V(G)$ and edge set $E(G)\cup \{e\}$.
A {\em claw} of $G$ is an induced subgraph isomorphic to the star
$K_{1,3}$.

A graph $G$ is {\em trivial} if it has only one vertex. Let $C_{o}(G)$ be
the number of odd components of $G$. The following is Tutte's $1$-factor theorem.

\begin{thm}[\cite{TTW}]\label{tutte} A graph $G$ has a $1$-factor
if and only if $C_{o}(G-X)\leq |X|$ for any $X \subseteq V(G)$.
\end{thm}

The following characterization and  connectivity  of $k$-factor-critical graphs
were obtained by O. Favaron \cite{F} and Q. Yu \cite{Y} independently.

\begin{lem}[\cite{F, Y}]\label{FY}
A graph $G$ is $k$-factor-critical if and only if $C_{o}(G-B)\leq |B|-k$ for
any $B \subseteq V(G)$ with $|B|\geq k$.
\end{lem}

\begin{lem}[\cite{F, Y}]\label{conn} If $G$ is $k$-factor-critical for some $1\leq k< n$ with $n+k$ even,
then $G$ is $k$-connected, $(k+1)$-edge-connected and $(k-2)$-factor-critical if $k\geq 2$.
\end{lem}

O. Favaron and M. Shi  \cite{FS} characterized minimally $k$-factor-critical graphs.

\begin{lem}[\cite{FS}]\label{minimal}
Let $G$ be a $k$-factor-critical graph. Then $G$ is minimal
if and only if for each $e=uv\in E(G)$, there exists $S_{e}\subseteq V(G)-\{u,v\}$ with $|S_{e}|=k$
such that every perfect matching of $G-S_{e}$ contains $e$.
\end{lem}

\begin{lem}[\cite{FS}]\label{n6} Let $G$ be a $k$-factor-critical graph of order $n>k+2$
and maximum degree $\Delta(G)=n-1$. Then $G$ is minimal if and only if $G$ contains one vertex of
degree $n-1$ and $n-1$ vertices of degree $k+1$.
\end{lem}

M.D. Plummer and A. Saito \cite{PS} obtained a necessary and sufficient condition of
$k$-factor-critical graphs.

\begin{thm}[\cite{PS}]\label{PS} Let $G$ be a graph of order $n$ and let $x$ and $y$ be a pair of nonadjacent
vertices of $G$ with $d_{G}(x)+ d_{G}(y)\geq n+k-1$. Then $G$ is
$k$-factor-critical if and only if $G\cup \{xy\}$ is $k$-factor-critical.
\end{thm}

\begin{cor}\label{n2} Let $G$ be a minimally $k$-factor-critical
graph of order $n$, where $n\geq k+5$.
If $\Delta(G)=n-2$, then there are at most two vertices with degree $n-2$
and such two vertices are not adjacent.
\end{cor}

\begin{proof} If $G$ has three vertices of degree $n-2$,
then two of them must be adjacent. So it suffices to
show that any two vertices with degree $n-2$ are not adjacent.
Suppose to the contrary that $d_{G}(u)=d_{G}(v)=n-2$ for $uv\in E(G)$.
Let $G'=G-uv$. Since $n\geq$ $k+5$,
we have $d_{G'}(u)+d_{G'}(v)=2n-6\geq n+k-1$. By Theorem \ref{PS}, $G'$ is
also $k$-factor-critical, contradicting that $G$ is minimally $k$-factor-critical graph.
\end{proof}

\section{A simple proof of Corollary \ref{Special}}

In this section, we give a different and brief method to reprove Corollary \ref{Special}.
We divide our proof into the two cases $k=n-4$ and $n-6$ for $k\geq 2$.

\begin{lem}[\cite{FS}] \label{n4} A graph $G$ of order $n\geq 6$ is $(n-4)$-factor-critical
if and only if it is claw-free and $\delta(G)\geq n-3$.
\end{lem}

\smallskip
\noindent{\bf Proof of Corollary \ref{Special} for $k=n-4$.}
By Lemma \ref{conn}, $\delta(G)\geq n-3$. To prove $\delta(G)=n-3$,
suppose to the contrary that $\delta(G)\geq n-2$.
Since $G$ is minimally $(n-4)$-factor-critical graph, for any $e=uv\in E(G)$,
$G-e$ is not $(n-4)$-factor-critical. Since $n=k+4\geq 6$ and $\delta(G-e)$$\geq n-3$,
by Lemma \ref{n4}, $G-e$ must contain a claw. Since $G$ is
$(n-4)$-factor-critical, $G$ is claw-free. Hence $u$ and $v$
must be two pendent vertices of the claw.
The third pendent vertex of the claw is not adjacent to $u$ and $v$.
So its degree is at most $n-3$, a contradiction.
~~~~~~~~~~~~~~~~~~~~~~~~~~~~~~~~~~~~~~~~~~~~~~~~~~~~~~~~~~~~~~~~~~
~~~~~~~~~~~~~~~~~~~~~~~~~~~~~~~~$\square$

\bigskip

\noindent{\bf Proof of Corollary \ref{Special} for $k=n-6$.}
Obviously, $\delta(G)\geq n-5$.
Suppose to the contrary that $\delta(G)\geq n-4$. That is, the non-neighborhood of any vertex in $G$ has at most three vertices.
Next we will obtain two claims.

{\textbf{Claim 1.}
For every $e=uv\in E(G)$, there exists $S_{e}\subseteq V(G)-\{u, v\}$
with $|S_{e}|=n-6$ such that $G-e-S_{e}$ is one of
Configurations $A1$, $A2$ and $A3$ as shown in Fig. 1.
(The vertices within a dotted box induce a connected subgraph and
the dotted edges indicate optional edges.)}

\begin{figure}[h]
\centering
\includegraphics[height=3.6cm,width=13.6cm]{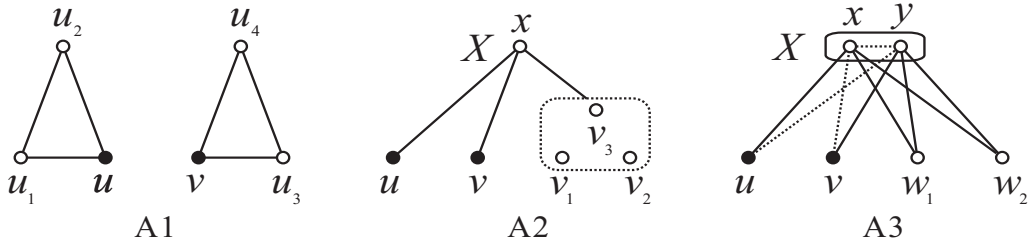}
\caption{\label{tu-1}The three configurations of $G'=G-e-S_{e}$.}
\end{figure}

Since $G$ is minimally $(n-6)$-factor-critical graph, by Lemma \ref{minimal},
for any given $e=uv\in E(G)$, there exists $S_{e}\subseteq V(G)-\{u, v\}$ with $|S_{e}|=n-6$
such that every perfect matching of $G-S_{e}$ contains $e$.
Then $G':=G-e-S_{e}$ has no perfect matching.
By Theorem \ref{tutte}, there exists $X\subseteq V(G')$ such that $C_{o}(G'-X)> |X|$.
By parity, $C_{o}(G'-X)\geq |X|+2$.
So $|X|+2\leq C_{o}(G'-X)\leq |V(G'-X)|=6-|X|$. Thus $|X|\leq 2$.
Since $G'+e$ has a perfect matching, $C_{o}(G'-X)=|X|+2$ and $u$ and $v$
belong separately to  distinct odd components of $G'-X$.
Since $\delta(G'+e)\geq 2$,  $G'$ has no isolated vertex. Now we discuss the following three cases
depending on $|X|$.

If $|X|=0$, then $G'$ must consist of two odd components isomorphic to $K_{3}$, and $e$ joins them. So $G'$ is $A1$.

If $|X|=1$, then $C_{o}(G'-X)=3$. If $G'-X$ has three trivial odd components, then
$G'+e$ has a pendent vertex, a contradiction. So $G'-X$ has exactly two trivial odd components and
one odd component with three vertices. Since $G'+e$ has no pendent vertex,
$e$ must join the two trivial odd components. So $G'$ is $A2$.

If $|X|=2$, then $C_{o}(G'-X)=4$. So $G'-X$ consists of exactly four trivial odd components,
two of which are joined by $e$. Let $X=\{x, y\}$.
If $x$ or $y\in \overline{N[u]}\cap \overline{N[v]}$, say $y\in \overline{N[u]}\cap \overline{N[v]}$,
then $ux, vx\in E(G')$ and $G'[w_{1}, w_{2}, y]$ is an odd component of $G'-\{x\}$.
Hence it is $A2$. Without loss of generality, assume that $ux, vy\in E(G')$.
Then $G'$ is $A3$. So Claim 1 holds.

Next we obtain some properties for the set of vertices of $G$ in both non-neighborhoods of end-vertices of an edge  from the three configurations.

{\textbf{Claim 2.} $(1)$ If $G-e-S_{e}$ is $A1$, then $|\overline{N[u]}\cap \overline{N[v]}|\leq 1$,
for $u_{1}, u_{2}\in N(u)$ but $u_{1}, u_{2}\notin N(v)$
and $u_{3}, u_{4}\in N(v)$ but $u_{3}, u_{4}\notin N(u)$;

\indent$(2)$ If $G-e-S_{e}$ is $A2$, then $|\overline{N[u]}\cap \overline{N[v]}|=3$ as
$\overline{N[u]}\cap \overline{N[v]}=\{v_{1}, v_{2}, v_{3}\}$;

\indent$(3)$ If $G-e-S_{e}$ is $A3$, then $|\overline{N[u]}\cap \overline{N[v]}|\geq 2$
as $ \overline{N[u]}\cap \overline{N[v]}$ contains a pair of non-adjacent vertices
 $\{w_{1}, w_{2}\}$.

By Claim 1, there are three cases to discuss, where contradictions always happen.

\smallskip
{\textbf{Case 1.} $G-e-S_{e}$ is $A1$.}

Consider edge $e'=uu_{1}$. By Claim 1, there exists $S_{e'}\subseteq V(G)-\{u, u_{1}\}$
with $|S_{e'}|=n-6$ such that $G-e'-S_{e'}$ is one of Configurations $A1$, $A2$ and $A3$.
Since $\overline{N[u_{1}]}=\{v, u_{3}, u_{4}\}$ and $uv\in E(G)$,
$\overline{N[u]}\cap \overline{N[u_{1}]}=\{u_{3}, u_{4}\}$.
By Claim 2 (1) and (2), $G-e'-S_{e'}$ is neither $A1$ nor $A2$. Since $u_{3}u_{4}\in E(G)$,
$\{u_{3}, u_{4}\}$ is not an independent set of $G$. So $G-e'-S_{e'}$ is not $A3$.
This is a contradiction to Claim 1.

\smallskip
{\textbf{Case 2.} $G-e-S_{e}$ is $A2$.}

Since $G-S_{e}$ has a perfect matching $M$,
without loss of generality, assume that $xv_{1}, v_{2}v_{3}\in M$.
Let $e'=ux\in E(G)$. Obviously, $\overline{N[u]}\cap \overline{N[x]}\subseteq \{v_{2}, v_{3}\}$.
By Claim 2 (2) and Case 1, $G-e'-S_{e'}$ is neither $A1$ nor $A2$ for any
$S_{e'}\subseteq V(G)-\{u, x\}$ with $|S_{e'}|=n-6$.
Because $v_{2}v_{3}\in E(G)$, $\{v_{2}, v_{3}\}$ is not an independent
set of $G$. Then $G-e'-S_{e'}$ is also not $A3$. This contradicts Claim 1.

\smallskip
{\textbf{Case 3.} $G-e-S_{e}$ is $A3$.}

Without loss of generality, assume that $ux, vy\in E(G)$.
Let $e'=ux\in E(G)$. Clearly, $w_{1}, w_{2}\in N(x)$ and $w_{1}, w_{2}\notin N(u)$.
Then $|\overline{N[u]}\cap \overline{N[x]}|\leq$ $1$.
By Claim 2 and Case 1, $G-e'-S_{e'}$ is not $A1$, $A2$ or $A3$ for any
$S_{e'}\subseteq V(G)-\{u, x\}$ with $|S_{e'}|=n-6$, which contradicts Claim 1.
~~~~~~~~~~~~~~~~~~~~~~~~~~~~~~~~~~~~~~~~~~~~~~~~~~~~~~~~~~~~~~~~
~~~~~~~~~~~~~~~~~~~~~~~~~~~~$\square$

\section{The minimum degree of minimally $(n-8)$-factor-critical graphs}

Going one step further, we confirm that Conjecture \ref{conj} is true for $k=n-8$.

\begin{thm}
If $G$ is minimally $(n-8)$-factor-critical graph of order $n\geq 10$,
then $\delta(G)=n-7$.
\end{thm}

\begin{proof}
By Lemma \ref{conn}, $\delta(G)\geq n-7$.
Suppose to the contrary that $\delta(G)\geq n-6$.

{\textbf{Claim 1.}
For every $e=uv\in E(G)$, there exists $S_{e}\subseteq V(G)-\{u, v\}$
with $|S_{e}|=n-8$ such that $G-e-S_{e}$ is one of Configurations $B1$ to $B8$
as shown in Fig. 2. (The vertices within a dotted box induce a connected subgraph.)}

\begin{figure}[h]
\centering
\includegraphics[height=8cm,width=15.2cm]{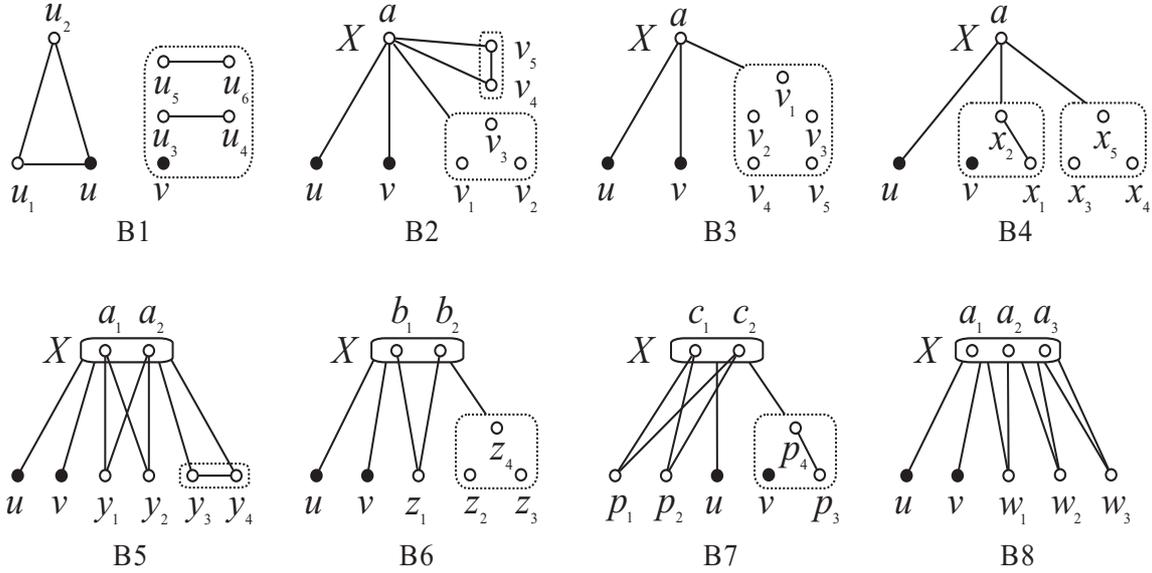}
\caption{\label{tu-2}The eight configurations of $G'=G-e-S_{e}$.}
\end{figure}

Since $G$ is minimally $(n-8)$-factor-critical graph, by Lemma \ref{minimal},
for any $e=uv\in E(G)$, there exists $S_{e}\subseteq V(G)-\{u, v\}$ with $|S_{e}|=n-8$
such that every perfect matching of $G-S_{e}$ contains $e$.
Then $G-e-S_{e}$ has no perfect matching. Let $G'=G-e-S_{e}$.
By Theorem \ref{tutte}, there exists $X\subseteq V(G')$ such that
$C_{o}(G'-X)> |X|$. By parity, $C_{o}(G'-X)\geq$ $|X|+2$.
So $|X|+2\leq$ $C_{o}(G'-X)\leq$ $|V(G'-X)|$$=8-|X|$. Thus $|X|\leq 3$.
Since $G'+e$ has a $1$-factor, $C_{o}(G'-X)=|X|+2$ and $u$ and $v$ belong respectively to
two distinct odd components of $G'-X$.
Moreover, $\delta(G-S_{e})\geq 2$. Then $G'+e=G-S_{e}$ has no pendent vertex and
$G'$ has no isolated vertex.

If $|X|=0$, then $G'$ has exactly two odd components,
one of which is $K_{3}$ and the other has five vertices.
Since $G'+e$ has a $1$-factor, $e$ joins the two odd components, and we may assume that
$u_{3}u_{4}, u_{5}u_{6}$ are two independent edges. So $G'$ is $B1$.

If $|X|=1$, then $C_{o}(G'-X)=3$. Let $X=\{a\}$. $G'-X$ has at most two trivial odd components
which are joined by $e$. Otherwise, $G'+e$ has a pendent vertex, a contradiction.
The other odd component has three or five vertices. So $G'$ is $B2$ or $B3$.
Specially, if $G'-X$ has exactly one trivial odd component, then the other two nontrivial
odd components are both with three vertices. Since $G'+e$ has no pendent vertex,
$e$ joins the trivial odd component and a nontrivial odd component with three vertices.
Besides, there must exist an edge joining $a$ and the
nontrivial odd component, otherwise, it is $B1$. Then $G'$ is $B4$.

If $|X|=2$, then $C_{o}(G'-X)=4$. Hence $G'-X$ has either four trivial odd components or
three trivial odd components and one nontrivial odd component with three vertices.
Since $G'+e$ has a $1$-factor, $e$ joins two of the four odd components of $G'-X$.
So $G'$ is $B5$, $B6$ or $B7$.

If $|X|=3$, then $C_{o}(G'-X)=5$. Thus $G'-X$ consists of exactly five trivial odd components,
two of which are joined by $e$. So $G'$ is $B8$.

\smallskip
For every $x\in V(G)$, there are at most five vertices of $G$ in $\overline{N[x]}$.
Then we can obtain the following claim by observing the eight configurations.

{\bf Claim 2.} $(1)$ If $G-e-S_{e}$ is $B1$, then $|\overline{N[u]}\cap \overline{N[v]}|\leq 3$.
Since $u_{1}, u_{2}\in N(u)$ but $u_{1}, u_{2}\notin N(v)$, $\overline{N[v]}$ has
at most three elements in $\overline{N[u]}$;

\indent$(2)$ If $G-e-S_{e}$ is $B2$ or $B3$,
then $|\overline{N[u]}\cap \overline{N[v]}|=5$;

\indent$(3)$ If $G-e-S_{e}$ is $B4$, then $3\leq |\overline{N[u]}\cap \overline{N[v]}|\leq 4$
as $\overline{N[u]}=\{x_{1}, x_{2}, x_{3}, x_{4}, x_{5}\}$
and $\{x_{3}, x_{4}, x_{5}\}$$\subseteq \overline{N[v]}$ but $x_{1}$ or $x_{2}\in N(v)$;

\indent$(4)$ If $G-e-S_{e}$ is $B5$ or $B6$, then $|\overline{N[u]}\cap \overline{N[v]}|\geq 4$;

\indent$(5)$ If $G-e-S_{e}$ is $B7$, then $2\leq |\overline{N[u]}\cap \overline{N[v]}|\leq 4$
as $\{p_{1}, p_{2}, p_{3}, p_{4}\}\subseteq$$\overline{N[u]}$
and $\{p_{1}, p_{2}\}$$\subseteq \overline{N[v]}$ but $p_{3}$ or $p_{4}\in N(v)$;

\indent$(6)$ If $G-e-S_{e}$ is $B8$, then $|\overline{N[u]}\cap \overline{N[v]}|\geq 3$
as  $\subseteq \overline{N[u]}\cap \overline{N[v]}$ contains
  an independent set $\{w_{1}, w_{2}, w_{3}\}$.

\smallskip
By Claim 1, there are eight cases to distinguish.

\smallskip
{\textbf{Case 1.} $G-e-S_{e}$ is $B1$.}

Since $G-S_{e}$ has a perfect matching $M$, $uv, u_{1}u_{2}\in M$. So,
without loss of generality, assume that $u_{3}u_{4}, u_{5}u_{6}\in M$.

Consider edge $e'=uu_{1}$. Clearly, $\overline{N[u_{1}]}=\{v, u_{3}, u_{4}, u_{5}, u_{6}\}$
and $\overline{N[u]}\cap \overline{N[u_{1}]}=$$\{u_{3}, u_{4},$ $u_{5}, u_{6}\}$.
By Claim 1, there exists $S_{e'}\subseteq V(G)-\{u, u_{1}\}$ with $|S_{e'}|=n-8$
such that $G-e'-S_{e'}$ is one of Configurations $B1$ to $B8$.
By Claim 2 (1) and (2), $G-e'-S_{e'}$ may not be $B1$, $B2$ or $B3$.
Furthermore, since $G[ \overline{N[u]}\cap \overline{N[u_{1}]} ]$ contains two independent edges,
by Claim 2 (4) and (6), $G-e'-S_{e'}$ can not be $B5$, $B6$ or $B8$.
Then $G-e'-S_{e'}$ would be $B4$ or $B7$.

Suppose that $G-e'-S_{e'}$ is $B4$.
Since $G[\overline{N[u_{1}]}]$ is a connected subgraph of $G$,
$u$ (resp. $u_{1}$) belongs to the trivial (resp. nontrivial) odd component of $B4-\{a\}$.
The odd component containing $u_{1}$ must be $K_{3}$. Otherwise, there is a vertex in
$\overline{N[u]}\cap \overline{N[u_{1}]}$ which is not adjacent to the other
three vertices, contradicting that $u_{3}u_{4}, u_{5}u_{6}$ are two independent edges.
But then $|\overline{N[u]}\cap \overline{N[u_{1}]}|= 3$, a contradiction.

So $G-e'-S_{e'}$ is $B7$. If $u_{1}$ is the trivial odd component of $B7-\{c_{1}, c_{2}\}$,
then $\{p_{1}, p_{2}\}=\{u_{3}, u_{5}\}$ or $\{p_{1}, p_{2}\}=\{u_{4}, u_{6}\}$,
say $\{p_{1}, p_{2}\}=\{u_{3}, u_{5}\}$. Hence $u_{4}$ or $u_{6}\in \{p_{3}, p_{4}\}$.
But $u_{3}u_{4}, u_{5}u_{6}\in E(G)$, a contradiction.
Then $u$ (resp. $u_{1}$) belongs to the trivial (resp. nontrivial)
odd component of $B7-\{c_{1}, c_{2}\}$. The odd component containing $u_{1}$
must be $K_{3}$. Otherwise, say $u_{1}p_{4}\notin E(G)$, so $p_{4}\in \{u_{4}, u_{6}\}$,
contradicting that $p_{1}p_{4}, p_{2}p_{4}\notin E(G)$.
But then $|\overline{N[u]}\cap \overline{N[u_{1}]}|\leq 3$, a contradiction.

\smallskip
{\textbf{Case 2.} $G-e-S_{e}$ is $B7$.}

We may assume that $uc_{1}, vp_{3}\in E(G)$. Since $G-S$ has a perfect matching, $p_{3}p_{4}\in E(G)$.
Let $e'=c_{1}p_{1}\in E(G)$. Obviously, $\overline{N[p_{1}]}=$$\{u, v, p_{2}, p_{3}, p_{4}\}$.
Then $\overline{N[c_{1}]}\cap \overline{N[p_{1}]}\subseteq \{v, p_{3}, p_{4}\}$ (see Fig. 3 (1)).
By Claim 1, there exists $S_{e'}\subseteq V(G)-\{c_{1}, p_{1}\}$
with $|S_{e'}|=n-8$ such that $G-e'-S_{e'}$ is one of Configurations $B1$ to $B8$.
Since $\{v, p_{3}, p_{4}\}$ is not an independent set of $G$, $G-e'-S_{e'}$ is not $B8$.
By Claim 2 (3), (5) and Case 1, $G-e'-S_{e'}$ would be $B4$ or $B7$.

If $G-e'-S_{e'}$ is $B4$, then $\overline{N[c_{1}]}\cap \overline{N[p_{1}]}=\{v, p_{3}, p_{4}\}=$
$\{x_{3}, x_{4}, x_{5}\}$. Hence $c_{1}$ (resp. $p_{1}$) belongs to the trivial (resp. nontrivial)
odd component of $B4-\{a\}$. Otherwise, $\overline{N[p_{1}]}=$ $\{x_{1}, x_{2}, x_{3}, x_{4}, x_{5}\}$,
so $u\in \{x_{1}, x_{2}\}$, but $uv\in E(G)$, a contradiction.
Moreover, the odd component of $B4-\{a\}$
containing $p_{1}$ must be $K_{3}$ as $|\overline{N[c_{1}]}\cap \overline{N[p_{1}]}|=3$.
Then $x_{1}$, $x_{2}\notin \{u, p_{2}\}$.
Since $p_{3}\in \{x_{3}, x_{4}, x_{5}\}$, $\{x_{1}, x_{2}\}\subseteq \overline{N[p_{3}]}$.
Then $\{c_{1}, u, p_{1}, p_{2}, x_{1}, x_{2}\}\subseteq \overline{N[p_{3}]}$.
Hence $d_{G}(p_{3})\leq n-7$, a contradiction (see Fig. 3 (2)).

\begin{figure}[h]
\centering
\includegraphics[height=4cm,width=14cm]{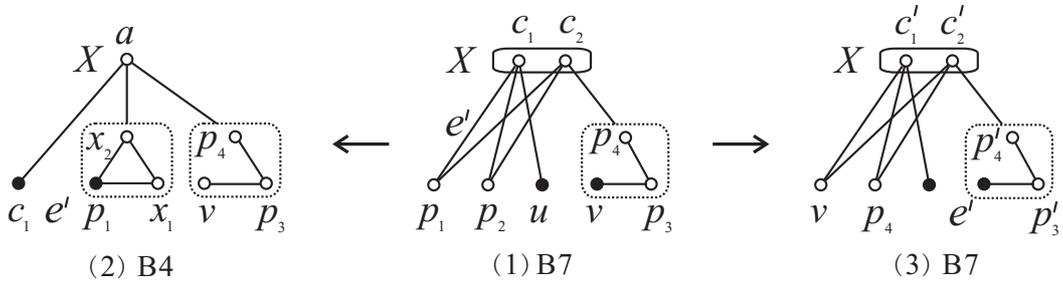}
\caption{\label{tu-3}$G-e'-S_{e'}$ is $B4$ and $G-e'-S_{e'}$ is $B7$.}
\end{figure}

So $G-e'-S_{e'}$ is $B7$, its vertices are relabelled by
$c_{1}', c_{2}', p_{1}', p_{2}', p_{3}', p_{4}', u', v'$,
Since $\{p_{1}', p_{2}'\}\subseteq$
$\overline{N[c_{1}]}\cap \overline{N[p_{1}]}\subseteq \{v, p_{3}, p_{4}\}$, $vp_{4}\notin E(G)$.
We may assume $p_{1}'=v$, $p_{2}'=p_{4}$. It is easy to see that
$\overline{N[p_{4}]}=$$\{u, v, c_{1}, p_{1}, p_{2}\}$. On the other hand,
$\overline{N[p_{4}]}=$$\{v, c_{1}, p_{1}, p_{3}', p_{4}'\}$. So $\{u, p_{2}\}=\{p_{3}', p_{4}'\}$.
But $p_{3}'p_{4}'\in E(G)$, $up_{2}\notin E(G)$, a contradiction (see Fig. 3 (3)).

\smallskip
{\textbf{Case 3.} $G-e-S_{e}$ is $B5$.}

Clearly, $\overline{N[y_{1}]}=$$\{u, v, y_{2}, y_{3}, y_{4}\}$.
Without loss of generality, assume that $ua_{1}\in E(G)$. Let $e'=y_{1}a_{1}\in E(G)$.
Then $\overline{N[a_{1}]}\cap \overline{N[y_{1}]}\subseteq \{v, y_{3}, y_{4}\}$.
By Claim 1, there exists $S_{e'}\subseteq V(G)-\{y_{1}, a_{1}\}$ with $|S_{e'}|=n-8$
such that $G-e'-S_{e'}$ is one of Configurations $B1$ to $B8$.
Since $y_{3}y_{4}\in E(G)$, $\{v, y_{3}, y_{4}\}$ is not an independent set of $G$.
So $G-e'-S_{e'}$ is not $B8$. By Claim 2 (3), Case 1 and Case 2, $G-e'-S_{e'}$ would be $B4$.

If $G-e'-S_{e'}$ is $B4$, then $\overline{N[a_{1}]}\cap \overline{N[y_{1}]}=$$\{v, y_{3}, y_{4}\}$,
which induces a connected subgraph of $G$. But $vy_{3}, vy_{4}\notin E(G)$, a contradiction.

\smallskip
{\textbf{Case 4.} $G-e-S_{e}$ is $B4$.}

Since $G-S_{e}$ has a perfect matching $M$, without loss of generality,
assume that $x_{1}x_{2}, ax_{3},x_{4}x_{5}\in M$.
Let $e'=ua\in E(G)$. Obviously, $\overline{N[u]}=\{x_{1}, x_{2}, x_{3}, x_{4}, x_{5}\}$.
Hence $\overline{N[a]}\cap \overline{N[u]}$$\subseteq \{x_{1}, x_{2}, x_{4}, x_{5}\}$.
By Claim 1, there exists $S_{e'}\subseteq V(G)-\{u, a\}$ with $|S_{e'}|=n-8$
such that $G-e'-S_{e'}$ is one of Configurations $B1$ to $B8$.
By Claim 2, Cases 1, 2 and 3, $G-e'-S_{e'}$ would be $B4$, $B6$ or $B8$.

Since $x_{1}x_{2}, x_{4}x_{5}\in E(G)$, $\{x_{1}, x_{2}, x_{4},$ $x_{5}\}$ has no subset,
which is an independent set with size three. So $G-e'-S_{e'}$ is not $B8$.
Because there are not edges joining $\{x_{1}, x_{2}\}$ and $\{x_{4}, x_{5}\}$,
for any $T\subseteq \{x_{1}, x_{2}, x_{4}, x_{5}\}$ with $|T|=3$,
$G[T]$ does not induce a component of $B4-\{a\}$ or $B6-\{b_{1}, b_{2}\}$.
So $G-e'-S_{e'}$ can not be $B4$ or $B6$. This is a contradiction to Claim 1.

\smallskip
{\textbf{Case 5.} $G-e-S_{e}$ is $B6$.}

Without loss of generality, we may assume that $b_{1}z_{2}\in E(G)$.
Let $e'=b_{1}z_{1}$. Clearly, $\overline{N[z_{1}]}=$$\{u, v, z_{2}, z_{3}, z_{4}\}$.
By Claim 1, there exists $S_{e'}\subseteq V(G)-\{b_{1}, z_{1}\}$
with $|S_{e'}|=n-8$ such that $G-e'-S_{e'}$ is one of Configurations $B1$ to $B8$.
Moreover, $b_{1}$, $b_{2}\notin \overline{N[u]}\cap \overline{N[v]}$,
otherwise, it is $B2$ or $B3$. Assume that $ub_{1}, vb_{2}\in E(G)$.
So $\overline{N[b_{1}]}\cap \overline{N[z_{1}]}\subseteq \{v, z_{3}, z_{4}\}$.
By Claim 2, Cases 1, 2 and 4, $G-e'-S_{e'}$ would be $B8$.

If $G-e'-S_{e'}$ is $B8$, then
$\overline{N[b_{1}]}\cap \overline{N[z_{1}]}=\{v, z_{3}, z_{4}\}$
and $\{v, z_{3}, z_{4}\}$ is an independent set of $G$.
So $z_{3}z_{4}\notin E(G)$. Thus $z_{2}z_{3}, z_{2}z_{4}\in E(G)$.
Since $\delta(G-S_{e})\geq 2$, $b_{2}z_{3}, b_{2}z_{4}\in E(G)$.
Now consider edge $e''=z_{2}z_{3}$.
We have $\overline{N[z_{3}]}=\{u, v, b_{1}, z_{1}, z_{4}\}$.
Since $b_{1}z_{2}, z_{2}z_{4}\in E(G)$, $\overline{N[z_{2}]}\cap \overline{N[z_{3}]}=\{u, v, z_{1}\}$.
By Claim 2, Cases 1, 2, 4 and $uv\in E(G)$, $G-e''-S_{e''}$
is not one of Configurations $B1$ to $B8$
for any $S_{e''}\subseteq V(G)-\{z_{2}, z_{3}\}$ with $|S_{e''}|=n-8$, a contradiction.

\smallskip
{\textbf{Case 6.} $G-e-S_{e}$ is $B2$ or $B3$.}

Since $G-S_{e}$ has a perfect matching $M$, without loss of generality, assume
that $av_{1}, v_{2}v_{3}, v_{4}v_{5}\in M$. Now consider edge $e'=ua$.
Then $\overline{N[u]}\cap \overline{N[a]}\subseteq \{v_{2}, v_{3}, v_{4}, v_{5}\}$.
By Claim 2 and Cases 1 to 5,
$G-e'-S_{e'}$ would be $B8$ for some $S_{e'}\subseteq V(G)-\{u, a\}$ with $|S_{e'}|=n-8$
only when $G-e-S_{e}$ is $B3$.
Since $v_{2}v_{3}, v_{4}v_{5}\in E(G)$, for any $T\subseteq$$\{v_{2}, v_{3}, v_{4}, v_{5}\}$
with $|T|=3$, $T$ is not an independent set of $G$. So $G-e'-S_{e'}$ is not $B8$.
This is a contradiction to Claim 1.

\smallskip
{\textbf{Case 7.} $G-e-S_{e}$ is $B8$.}

Since $\delta(G-S_{e})\geq 2$, without loss of generality,
assume that $w_{1}a_{1}, w_{1}a_{2}\in E(G)$.
Let $e'=w_{1}a_{1}\in E(G)$ and $S_{e'}\subseteq V(G)-\{w_{1}, a_{1}\}$
with $|S_{e'}|=n-8$ satisfying Claim 1.
Then $\{u, v, w_{2}, w_{3}\}\subseteq \overline{N[w_{1}]}$.
By Claim 2 and Cases 1 to 6, $G-e'-S_{e'}$ would be $B8$.
Then it suffices to show that there is an independent set $S_{0}'$ with size three
in $G-S_{e'}$ and every vertex in $S_{0}'$ is not adjacent to $w_{1}$ and $a_{1}$.
Since $S_{0}'\subseteq \overline{N[w_{1}]}$,
$w_{2}$ or $w_{3}\in S_{0}'$, say $w_{2}\in S_{0}'$.
Then $w_{2}a_{2}, w_{2}a_{3}\in E(G)$ as $d_{G-S_{e}}(w_{2})\geq 2$.

\smallskip
{\textbf{Subcase 7.1.} $w_{1}a_{3}\notin E(G)$.}

Clearly, $\overline{N[w_{1}]}=\{u, v, w_{2}, w_{3}, a_{3}\}$.
So $a_{3}\notin S_{0}'$ and $w_{3}\in S_{0}'$. Then $w_{3}a_{2}, w_{3}a_{3}\in E(G)$.
Thus $S_{0}'$ is either $\{u, w_{2}, w_{3}\}$ or $\{v, w_{2}, w_{3}\}$,
say $\{u, w_{2}, w_{3}\}$. Now consider edge $e''=w_{1}a_{2}$.
Then $\overline{N[a_{2}]}\cap \overline{N[w_{1}]}\subseteq \{u, v, a_{3}\}$.
By Claim 2, Cases 1, 2, 4 and $uv\in E(G)$, $G-e''-S_{e''}$ is not one of Configurations $B1$ to $B8$
for any $S_{e''}\subseteq V(G)-\{w_{1}, a_{2}\}$ with $|S_{e''}|=n-8$, which contradicts Claim 1.

\smallskip
{\textbf{Subcase 7.2.} $w_{1}a_{3}\in E(G)$.}

Let $e''=w_{1}a_{2}\in E(G)$ and $S_{e''}\subseteq V(G)-\{w_{1}, a_{2}\}$
with $|S_{e''}|=n-8$ satisfying Claim 1.
We denote the independent set with size three in $G-S_{e''}$
by $S_{0}''$ and every vertex in $S_{0}''$ is not adjacent to $w_{1}$ and $a_{2}$.
Then $w_{3}\in S_{0}''$. So $w_{3}a_{1}, w_{3}a_{3}\in E(G)$.
Hence $w_{3}\notin S_{0}'$. Thus $u$ or $v\in S_{0}'$, say $u\in S_{0}'$.
Then $S_{0}'=\{u, w_{2}, w\}$, where
$w\in \overline{N[w_{1}]}$$\cap$$\overline{N[a_{1}]}$.
Since $w_{2}\notin S_{0}''$, $u$ or $v\in S_{0}''$.

If $u\in S_{0}''$, then $S_{0}''=\{u, w_{3}, w\}$, where
$w\in \overline{N[w_{1}]}$$\cap$$\overline{N[a_{2}]}$.
Thus $\{u, w_{1}, w_{2}, w_{3}, a_{1}, a_{2}\}$ $\subseteq$ $\overline{N[w]}$.
So $d_{G}(w)\leq n-7$, a contradiction.

If $v\in S_{0}''$, then $S_{0}''=\{v, w_{3}, w\}$, where
$w\in \overline{N[w_{1}]}$$\cap$$\overline{N[a_{2}]}$.
Thus $\{u, v, w_{1}, w_{2}, w_{3}, a_{1}, a_{2}\}$ $\subseteq$ $\overline{N[w]}$.
So $d_{G}(w)\leq n-8$, a contradiction.

\smallskip
Combining Cases 1 to 7, we complete the proof.
\end{proof}

\section{Some properties of minimally $(n-6)$-factor-critical graphs}

In this section, we obtain that
every minimally $(n-6)$-factor-critical graph of order $n$ has
at most $n-\Delta(G)$ vertices with the maximum degree $\Delta(G)$
for $n-4\leq \Delta(G)\leq n-1$.

By Lemma \ref{n6}, for any minimally $(n-6)$-factor-critical graph of order $n$,
$G$ has only one vertex of degree $n-1$ and
$n-1$ vertices of degree $n-5$ when $\Delta(G)=n-1$. So we consider
the cases of $n-4\leq \Delta(G)\leq n-2$.

\begin{lem}\label{four}
Let $G$ be minimally $(n-6)$-factor-critical graph of order $n\geq 8$.
For every $e=uv\in E(G)$ with $d_{G}(u)\geq n-4$, $d_{G}(v)\geq n-4$, there exists
$S_{e}\subseteq V(G)-\{u, v\}$ with $|S_{e}|=$$n-6$ such that $G-e-S_{e}$ is one of
Configurations $C1$, $C2$, $C3$ and $C4$ as shown in Fig. 4.
(The dotted edge indicates an optional edge.)

\begin{figure}[h]
\centering
\includegraphics[height=4cm,width=15cm]{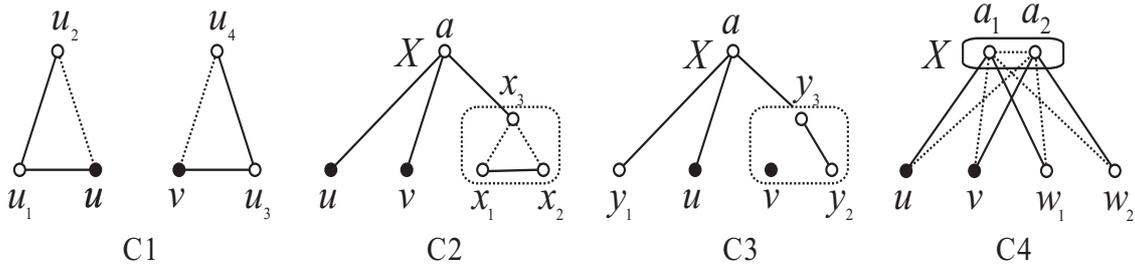}
\caption{\label{tu-4}The four configurations of $G'=G-e-S_{e}$.}
\end{figure}
\end{lem}

\begin{proof}
Since $G$ is minimally $(n-6)$-factor-critical graph, for every $e=uv\in E(G)$
with $d_{G}(u)\geq n-4$, $d_{G}(v)\geq n-4$, there exists $S_{e}\subseteq V(G)-\{u, v\}$
with $|S_{e}|=n-6$ such that every perfect matching of $G-S_{e}$ contains $e$
by Lemma \ref{minimal}. Then $G-e-S_{e}$ has no $1$-factor. Let $G'=G-e-S_{e}$.
By Theorem \ref{tutte}, there exists $X\subseteq V(G')$ such that $C_{o}(G'-X)> |X|$.
By parity, $C_{o}(G'-X)\geq |X|+2$.
So $|X|+2\leq C_{o}(G'-X)\leq |V(G'-X)|=6-|X|$. Then $|X|\leq 2$.
Since $G'+e$ has a $1$-factor, $C_{o}(G'-X)=|X|+2$ and $u$ and $v$ belong
respectively to two distinct odd components of $G'-X$.
Furthermore, $d_{G-e}(u)\geq n-5$ and $d_{G-e}(v)\geq n-5$. The vertices of $G-e$
distinct from $u$ and $v$ have the same degree as $G$.
So $\delta(G-e)=\delta(G)=n-5$ by Corollary \ref{Special}. Then $\delta(G-e-S_{e})=\delta(G')\geq 1$.
Thus $G'$ has no isolated vertex.

If $|X|=0$, then $G'$ has exactly two odd components, each of which has three vertices.
Since $G'+e$ has a $1$-factor, $e$ joins the two odd components. So $G'$ is $C1$.

If $|X|=1$, then $C_{o}(G'-X)=3$. Let $X=\{a\}$.
$G'-X$ has either three trivial odd components and an even component
with two vertices or two trivial odd components and an odd component with three vertices.
If $e$ joins two trivial odd components, then $G'$ is $C2$.
If $e$ joins a trivial odd component and the odd component with three vertices,
then $G'$ is $C3$.

If $|X|=2$, then $C_{o}(G'-X)=4$. So $G'-X$ consists of exactly four trivial odd components,
two of which are joined by $e$.
Let $X=\{a_{1}, a_{2}\}$. Then $a_{1}$, $a_{2}\notin \overline{N[u]}\cap \overline{N[v]}$,
otherwise, it is $C2$ or $C3$. So $G'$ is $C4$.
\end{proof}

\smallskip
Since $d_{G}(u)\geq n-4$ and $d_{G}(v)\geq n-4$, there are at most three vertices of
$G$ in each of $\overline{N[u]}$ and $\overline{N[v]}$.
By a close inspection of the four configurations of Lemma \ref{four},
we can easily obtain some properties, which plays important roles in the proofs of
Theorems~\ref{thm 5.3}, \ref{thm 5.4} and \ref{thm 5.5}.
\smallskip

\begin{prop}\label{prop}
Let $G, e, S_{e}$ be of Lemma \ref{four}.
If $G-e-S_{e}$ is one of Configurations $C1$, $C2$, $C3$ and $C4$, then
the following statements hold:

(1) If $G-e-S_{e}$ is $C1$, then
$|\overline{N[u]}\cap \overline{N[v]}|\leq 2$, $n-4\leq d_{G}(u), d_{G}(v)\leq n-3$ and
$|N(u)\cap N(v)|\leq n-6$.

(2) If $G-e-S_{e}$ is $C2$, then $|\overline{N[u]}\cap \overline{N[v]}|=3$ and
$d_{G}(u)=d_{G}(v)=n-4$.

(3) If $G-e-S_{e}$ is $C3$, then $1 \leq|\overline{N[u]}\cap \overline{N[v]}|\leq 2$.
Moreover, $d_{G}(u)=n-4$, $d_{G}(v)\geq n-4$ and $d_{G}(y_{1})=n-5$.

(4) If $G-e-S_{e}$ is $C4$, then $2 \leq|\overline{N[u]}\cap \overline{N[v]}|\leq 3$
and $n-4\leq d_{G}(u), d_{G}(v)\leq n-3$. Moreover, $\{w_{1}, w_{2}\}$ is an independent set of $G$.
\end{prop}

\smallskip
If $n$ is odd, then both $n-2$ and $n-4$ are odd. Thus the total number of vertices of degree
$n-2$ or $n-4$ is even.
If $n$ is even, then both $n-3$ and $n-5$ are odd. So the total number of vertices of degree
$n-3$ or $n-5$ is even. Thus the total number of vertices of degree
$n-2$ or $n-4$ is also even.
Therefore, $G$ contains even total number of vertices of
degree $n-2$ or $n-4$.

\begin{thm}\label{thm 5.3}
Let $G$ be minimally $(n-6)$-factor-critical graph of order $n\geq 8$.
If $\Delta(G)=n-2$, then $G$ has at most two vertices with degree
$n-2$ and the two vertices are not adjacent. In particular, if $G$ has exactly two vertices with degree
$n-2$, then the other vertices of $G$ have degree $n-5$. If $G$ has one vertex with degree
$n-2$, then there is only one vertex with degree $n-4$ and the other
vertices of $G$ have degree $n-5$.
\end{thm}

\begin{proof}
Firstly, by Corollary \ref{n2}, $G$ has at most two vertices with degree $n-2$ and
the two vertices are not adjacent. Let $d_{G}(u)=d_{G}(v)=n-2$ and $uv\notin E(G)$.
If $G$ has a vertex $x$ with degree $n-3$ or $n-4$,
then $ux, vx\in E(G)$. Consider edge $ux$. By Lemma \ref{four}, there exists
$S\subseteq V(G)-\{u, x\}$ with $|S|=n-6$ such that $G-ux-S$ is
one of Configurations $C1$, $C2$, $C3$ and $C4$.
Then, by Proposition \ref{prop}, $G-ux-S$ would be $C3$ only when $d_{G}(x)=n-4$.
Hence $x$ (resp. $u$) belongs to the trivial (resp. nontrivial) odd component of $C3-\{a\}$.
So $y_{1}=v$ and $d_{G}(v)=d_{G}(y_{1})=n-5$, a contradiction.
Thus all vertices of $G-\{u,v\}$ have degree $n-5$.

If $G$ has only one vertex $u$ with degree $n-2$, then $u$ is adjacent to every
vertex of $G$ except one. Suppose that $G$ has three vertices with degree $n-4$,
say $d_{G}(v_{1})=d_{G}(v_{2})=d_{G}(v_{3})=n-4$. Then we may assume that $uv_{1}, uv_{2}\in E(G)$.
Let $e_{1}=uv_{1}\in E(G)$ and $S_{1}\subseteq V(G)-\{u, v_{1}\}$
with $|S_{1}|=n-6$ satisfying Lemma \ref{four}. By Proposition \ref{prop}, $G-e_{1}-S_{1}$ is only $C3$.
Then $v_{1}$ (resp. $u$) belongs to the trivial (resp. nontrivial) odd component of $C3-\{a\}$.
Assume another trivial odd component of $C3-\{a\}$
is spanned by $\{w\}$. So $uw$, $wv_{1}\notin E(G)$. Hence $uv_{3}\in E(G)$.
Let $e_{2}=uv_{2}\in E(G)$ and $e_{3}=uv_{3}\in E(G)$. Similar discussion above,
we have $wv_{2}, wv_{3}\notin E(G)$. Then $\overline{N[w]}=$$\{u, v_{1}, v_{2}, v_{3}\}$.
Therefore, $G[\{u, v_{2}, v_{3}\}]$ is the nontrivial odd component of $G-e_{1}-S_{1}$.
So $v_{1}v_{2}, v_{1}v_{3}\notin E(G)$. Similarly, we have $v_{2}v_{3}\notin E(G)$.
But then $G-S_{1}$ has no $1$-factor, a contradiction.
Therefore, if $d_{G}(u)=n-2$, $G$ has only one vertex with degree $n-4$.

Suppose that $d_{G}(u)=n-2$ and $d_{G}(v_{1})=n-4$. If $G$ has a vertex $y$
with degree $n-3$, then $uy\notin E(G)$. Otherwise, by Proposition \ref{prop},
$G-uy-S'$ is not one of Configurations $C1$, $C2$, $C3$ and $C4$ for any
$S'\subseteq V(G)-\{u, y\}$ with $|S'|=n-6$, a contradiction. So $uv_{1}\in E(G)$.
However, from the above discussion, $G-uv_{1}-S_{1}$ is only $C3$.
Thus $y$ belongs to the trivial odd component of $C3-\{a\}$.
So $d_{G}(y)=n-5$, a contradiction.
Thus all vertices of $G-\{u, v_{1}\}$ have degree $n-5$.
\end{proof}

\begin{thm}\label{thm 5.4}
Let $G$ be minimally $(n-6)$-factor-critical graph of order $n\geq 9$.
If $\Delta(G)=n-3$, then $G$ has at most three vertices with degree
$n-3$. In particular, if $d_{G}(u)=d_{G}(v)=d_{G}(w)=n-3$, then
$uv, uw, vw\notin E(G)$ and the other vertices of $G$ have degree $n-5$.
\end{thm}

\begin{proof}
If $G$ has four vertices with degree $n-3$, then two of them must be adjacent.
So it suffices to show that any three vertices with degree $n-3$ are not adjacent each other.
Assume that $d_{G}(u)=d_{G}(v)=d_{G}(w)=n-3$.
Suppose to the contrary that there is at least one pair of adjacent vertices among
$\{u, v, w\}$, say $uv\in E(G)$. Let $e=uv\in E(G)$. By Lemma \ref{four}, there exists
$S\subseteq V(G)-\{u, v\}$ with $|S|=n-6$ such that $G-e-S$ is
one of Configurations $C1$, $C2$, $C3$ and $C4$.

\smallskip
{\textbf{Case 1.} $uv\in E(G), uw, vw\notin E(G)$.}

Then $|\overline{N[u]}\cap \overline{N[v]}|\leq 2$ and $|N(u)\cap N(v)|\geq n-5$.
By Proposition \ref{prop}, $G-e-S$ would be $C4$.

If $G-e-S$ is $C4$, then $|\overline{N[u]}\cap \overline{N[v]}|=2$ and
$w\in \overline{N[u]}\cap \overline{N[v]}$.
Since $w$ is adjacent to every vertex in $V(G)-\{u, v, w\}$, there is not an
independent set with size two containing $w$ and a vertex in $V(G)-\{u, v, w\}$.
So $G-e-S$ is not $C4$. This contradicts Lemma \ref{four}.

\smallskip
{\textbf{Case 2.} $uv, vw\in E(G), uw\notin E(G)$.}

Then $|\overline{N[u]}\cap \overline{N[v]}|\leq 1$ and $|N(u)\cap N(v)|\geq n-6$.
By Proposition \ref{prop}, $G-e-S$ would be $C1$.

If $G-e-S$ is $C1$, then $|N(u)\cap N(v)|=n-6$ and $|\overline{N[u]}\cap \overline{N[v]}|=0$.
Hence $S=$ $N(u)\cap N(v)$.
Moreover, $w$ belongs to the odd component of $G-e-S$ containing $v$.
But $w$ is adjacent to at least one vertex in $N(u)\setminus N[v]$ as $d_{G}(w)=n-3$.
Thus $G-e-S$ is connected which can not be $C1$. This is a contradiction to Lemma \ref{four}.

\smallskip
{\textbf{Case 3.} $uv, uw, vw\in E(G)$.}

Then $|\overline{N[u]}\cap \overline{N[v]}|\leq 2$ and $|N(u)\cap N(v)|\leq n-4$.
We discuss the three subcases.

\smallskip
{\textbf{Subcase 3.1.} $|\overline{N[u]}\cap \overline{N[v]}|=2$.}

Let $\overline{N[u]}\cap \overline{N[v]}=\{x, y\}$. Obviously, $|N(u)\cap N(v)|=n-4$.
By Proposition \ref{prop}, $G-e-S$ would be $C4$.

If $G-e-S$ is $C4$, then $\{x, y\}$ is an
independent set of $C4$. Thus $xy\notin E(G)$.
If $wx\in E(G)$, $wy\notin E(G)$ or $wx\notin E(G)$, $wy\in E(G)$, we consider
edge $uw$ or $vw$ the same as Subcase 3.2.
If $wx$, $wy\in E(G)$, we consider
edge $uw$ or $vw$ the same as Subcase 3.3.
Assume that $wx$, $wy\notin E(G)$.
Since $\delta(G)=n-5$, $x$ and $y$ are adjacent to every vertex in
$V(G)-\{u, v, w, x, y\}$.
Let $S'\subseteq V(G)-\{u, v, w, x, y\}$ and $|S'|=n-6$. Then $G-S'$
has no $1$-factor, contradicting that $G$ is $(n-6)$-factor-critical.

\smallskip
{\textbf{Subcase 3.2.} $|\overline{N[u]}\cap \overline{N[v]}|=1$.}

Let $\overline{N[u]}\cap \overline{N[v]}=\{x\}$.
Clearly, $|N(u)\cap N(v)|=n-5$. By Proposition \ref{prop}, $G-e-S$ can not be
Configurations $C1$, $C2$, $C3$ or $C4$, which contradicts Lemma \ref{four}.

\smallskip
{\textbf{Subcase 3.3.} $|\overline{N[u]}\cap \overline{N[v]}|=0$.}

Clearly, $|N(u)\cap N(v)|=n-6$. Next, we consider the cardinalities of
$\overline{N[u]}\cap \overline{N[w]}$ and $\overline{N[v]}\cap \overline{N[w]}$.

{\bf Subcase 3.3.1.}
$|\overline{N[u]}\cap \overline{N[w]}|=1$ or $|\overline{N[v]}\cap \overline{N[w]}|=1$.

We consider edge $uw$ or $vw$ the same as Subcase 3.2.

{\bf Subcase 3.3.2.}
$|\overline{N[u]}\cap \overline{N[w]}|=2$ or $|\overline{N[v]}\cap \overline{N[w]}|=2$.

Without loss of generality, assume that $|\overline{N[u]}\cap \overline{N[w]}|=2$.
Then $|\overline{N[v]}\cap \overline{N[w]}|=0$ and $|N(u)\cap N(w)|=n-4$.
Let $\overline{N[u]}\cap \overline{N[w]}=\{x, y\}$.
By Proposition \ref{prop}, $G-e-S$ would be $C1$.
Then $S=N(u)\cap N(v)$ and $G[\{v, x, y\}]$ is an odd component of $C1$.
Since $G-S$ has a $1$-factor, $xy\in E(G)$. However, by Proposition \ref{prop},
$G-uw-S'$ would be $C4$ for some $S'\subseteq V(G)-\{u, w\}$ with $|S'|=n-6$.
Hence $\{x, y\}$ is an independent set of $G$. So $xy\notin E(G)$, a contradiction.

{\bf Subcase 3.3.3.}
$|\overline{N[u]}\cap \overline{N[w]}|=0$ and $|\overline{N[v]}\cap \overline{N[w]}|=0$.

By Proposition \ref{prop}, $G-e-S$ would be $C1$. Hence $S=N(u)\cap N(v)$.

Let $S_{1}'=N(u)\setminus N[v]$, $S_{2}'=N(v)\setminus N[u]$, $S_{3}'=N(u)\setminus N[w]$
and $S_{0}=N(u)\cap N(v) \cap N(w)$. Then $|S_{1}'|=|S_{2}'|=|S_{3}'|=2$ and
$|S_{0}|=n-|S_{1}'|-|S_{2}'|-|S_{3}'|-3=n-9$.
If $G-e-S$ is $C1$, then $S=N(u)\cap N(v)$=$S_{3}'\cup S_{0} \cup \{w\}$ and
there are not edges joining $S_{1}'$ and $S_{2}'$.
Let $e_{1}=uw$, $S_{1}\subseteq V(G)-\{u, w\}$ with $|S_{1}|=n-6$
and $e_{2}=vw$, $S_{2}\subseteq V(G)-\{v, w\}$ with $|S_{2}|=n-6$.
By Proposition \ref{prop}, both $G-e_{1}-S_{1}$ and $G-e_{2}-S_{2}$ are also $C1$.
Then $S_{1}=N(u)\cap N(w)$=$S_{1}'\cup S_{0} \cup \{v\}$
and there are not edges joining $S_{2}'$ and $S_{3}'$.
Moreover, $S_{2}=N(v)\cap N(w)$=$S_{2}'\cup S_{0} \cup \{u\}$
and there are not edges joining $S_{1}'$ and $S_{3}'$.
Therefore, there are not edges joining $S_{1}'$, $S_{2}'$ and $S_{3}'$ each other.
Since $\delta(G)=n-5$ and $G[S_{1}'\cup S_{2}'\cup S_{3}']$ has a $1$-factor,
every vertex in $S_{i}'$ is incident with $S_{0}$ at least $n-8$ edges for $i=1, 2, 3$.
But $|S_{0}|=n-9< n-8$, a contradiction.

\smallskip
Therefore, if $d_{G}(u)=d_{G}(v)=d_{G}(w)=n-3$, then $uv, uw, vw\notin E(G)$.

\smallskip
Now suppose that $G$ has two vertices $x, y$ with degree $n-4$. Then every vertex in
$\{u, v, w\}$ is adjacent to $x$ and $y$.
Let $e''= ux\in E(G)$ and $S''\subseteq V(G)-\{u, x\}$ with $|S''|=n-6$ satisfying Lemma \ref{four}.
Then $|\overline{N[u]}\cap \overline{N[x]}|=0$ and $|N(u)\cap N(x)|=n-7$.
By Proposition \ref{prop}, $G-e''-S''$ would be $C1$.
Thus $G[\{x, v, w\}]$ must be an odd component of $C1$.
Since $vw\notin E(G)$,  $G-S''$ has no $1$-factor, a contradiction.
So $G-e''-S''$ is not $C1$. This contradicts Lemma \ref{four}.
Therefore all vertices of $G-\{u, v, w\}$ have degree $n-5$.
\end{proof}

\begin{thm}\label{thm 5.5}
Let $G$ be minimally $(n-6)$-factor-critical graph of order $n\geq 11$.
If $\Delta(G)=n-4$, then $G$ has at most four vertices with degree
$n-4$ and the other vertices of $G$ have degree $n-5$.
\end{thm}

\begin{proof}
Since $G$ has an even number of vertices with degree $n-4$,
suppose that $G$ has six vertices with degree $n-4$, say $\{u, v, w, x, y, z\}$.
Let $S_{0}=\{u, v, w, x, y, z\}$. Every vertex in $S_{0}$
is adjacent to at most $n-6$ vertices in $V(G)-S_{0}$. Then $\delta(G[S_{0}])\geq 2$.
So $G[S_{0}]$ must contain a cycle. Moreover, $G[S_{0}]$ has a $1$-factor.
Thus we can always find a path $P$ with length three in $G[S_{0}]$.
To prove the theorem, we need only to show that $G$ does not contain such a path $P$
with length three, in which each vertex has degree $n-4$.

Suppose to the contrary that $G$ contains a path $P=uvxy$, in which each vertex has degree $n-4$.
Let $e=uv\in E(G)$. By Lemma \ref{four}, there exists
$S\subseteq V(G)-\{u, v\}$ with $|S|=n-6$ such that $G-e-S$ is
one of Configurations $C1$, $C2$, $C3$ and $C4$.
We discuss the two cases depending on $|\overline{N[u]}\cap \overline{N[v]}|$.

\smallskip
{\textbf{Case 1.} $|\overline{N[u]}\cap \overline{N[v]}|=1$.}

Clearly, $\overline{N[u]}\cap \overline{N[v]}=\{y\}$ and $|N(u)\cap N(v)|=n-7$.
Let $N(u)\setminus N[v]=\{u_{1}, u_{2}\}$ and $N(v)\setminus (N[u]\cup \{x\})=$ $\{v_{1}\}$.
By Proposition \ref{prop}, $G-e-S$ would be $C1$ or $C3$.

\smallskip
{\textbf{Subcase 1.1.} $G-e-S$ is $C1$.}

Then $N(u)\cap N(v)\subseteq S$
and $C1$ is one of Configurations $(a)$ and $(b)$ (see Fig. 5).

\begin{figure}[h]
\centering
\includegraphics[height=3cm,width=10.8cm]{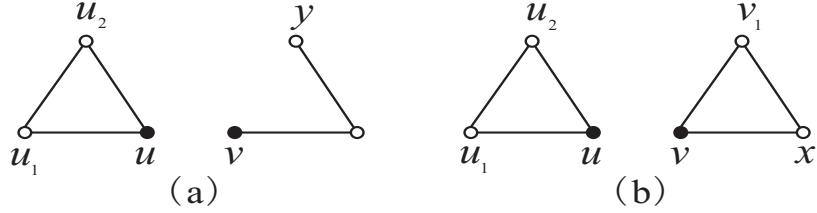}
\caption{\label{tu-5}The two configurations of $C1$.}
\end{figure}

$C1$ is $(a)$. Since $d_{G}(y)=n-4$, $yu_{1}$ or $yu_{2}\in E(G)$, a contradiction.

$C1$ is $(b)$. Then $S=N(u)\cap N(v)\cup \{y\}$. Hence $x$ is adjacent to every vertex in
$V(G)-$$\{u, u_{1}, u_{2}\}$. Now consider edge $e'=vx$.
Clearly, $\overline{N[v]}\cap \overline{N[x]}=$$\{u_{1}, u_{2}\}$
and $|N(v)\cap N(x)|=n-6$. By Proposition \ref{prop},
there exists $S'\subseteq V(G)-\{v, x\}$ with $|S'|=n-6$ such that
$G-e'-S'$ would be $C1$ or $C3$.

If $G-e'-S'$ is $C1$, then $S'=N(v)\cap N(x)$. Hence $u_{1}$ and $u_{2}$ belong respectively
to two distinct odd components of $G-e'-S'$. But $u_{1}u_{2}\in E(G)$, a contradiction.

If $G-e'-S'$ is $C3$, then $v$ (resp. $x$) belongs to the trivial (resp. nontrivial)
odd component of $C3-\{a\}$. Otherwise, $\{y_{1}, y_{2}, y_{3}\}=$$\{u, u_{1}, u_{2}\}$, a contradiction.
Since $\overline{N[v]}=$$\{y, u_{1}, u_{2}\}$ and $xy\in E(G)$, $y_{1}\in \{u_{1}, u_{2}\}$,
say $y_{1}=u_{1}$. So $\{y_{2}, y_{3}\}=$$\{u_{2}, y\}$.
Then $u_{1}u_{2}=y_{1}y_{2}\in E(G)$ or $u_{1}u_{2}=y_{1}y_{3}\in E(G)$, a contradiction.

\smallskip
{\textbf{Subcase 1.2.} $G-e-S$ is $C3$.}

Obviously, $y_{1}=y$. If $u$ belongs to the trivial odd component of $C3-\{a\}$,
then $G[\{v, x, v_{1}\}]$ is the nontrivial odd component of $C3-\{a\}$.
But $xy\in E(G)$, a contradiction. If $v$ belongs to the trivial odd component of $C3-\{a\}$,
then $G[\{u, u_{1}, u_{2}\}]$ is the nontrivial odd component of $C3-\{a\}$.
But $yu_{1}$ or $yu_{2}\in E(G)$ as $d_{G}(y)=n-4$, a contradiction.

\smallskip
Therefore, $G-e-S$ is not $C1$ or $C3$ when $|\overline{N[u]}\cap \overline{N[v]}|=1$.

\smallskip
{\textbf{Case 2.} $|\overline{N[u]}\cap \overline{N[v]}|=2$.}

Let $\overline{N[u]}\cap \overline{N[v]}=\{y, v_{1}\}$ and
$N(u)\setminus N[v]=u_{1}$. Then $\overline{N[u]}=\{x, y, v_{1}\}$ and
$\overline{N[v]}=$ $\{u_{1}, y, v_{1}\}$. Obviously, $|N(u)\cap N(v)|=n-6$.
By Proposition \ref{prop}, $G-e-S$ would be $C1, C3$ or $C4$.

\smallskip
{\textbf{Subcase 2.1.} $G-e-S$ is $C1$.}

Then $S=N(u)\cap N(v)$.
So $G[\{v, x, y\}]$ and $G[\{u, u_{1}, v_{1}\}]$ are two odd components of $C1$.
But $yu_{1}$ or $yv_{1}\in E(G)$ as $d_{G}(y)=n-4$, a contradiction.

\smallskip
{\textbf{Subcase 2.2.} $G-e-S$ is $C3$.}

Let $e'=vx\in E(G)$. By Lemma \ref{four}, there exists
$S'\subseteq V(G)-\{v, x\}$ with $|S'|=n-6$ such that $G-e'-S'$ is
one of Configurations $C1$, $C2$, $C3$ and $C4$.

\smallskip
{\textbf{Subcase 2.2.1.} $u$ (resp. $v$) belongs to the trivial (resp. nontrivial)
odd component of $C3-\{a\}$.}

Then $\{y_{1}, y_{2}, y_{3}\}=\{x, y, v_{1}\}$. Since $xy\in E(G)$, $y_{1}=v_{1}$.
Hence $xv_{1},$ $yv_{1}\notin E(G)$.
So $u_{1}v_{1}\in E(G)$ as $d_{G}(v_{1})\geq n-5$ and $yu_{1}\in E(G)$ as $d_{G}(y)=n-4$.

{\bf (2.2.1.1)} $|\overline{N[v]}\cap \overline{N[x]}|=2$.

Then $\overline{N[v]}\cap \overline{N[x]}=\{u_{1}, v_{1}\}$ and $|N(v)\cap N(x)|=n-6$.
By Proposition \ref{prop}, $G-e'-S'$ would be $C1, C3$ or $C4$.

If $G-e'-S'$ is $C1$, then $S'=$$N(v)\cap N(x)$. Hence $u_{1}$ and $v_{1}$ belong
respectively to two distinct odd components of $C1$. But $u_{1}v_{1}\in E(G)$.
Then $G-e'-S'$ is connected, which is a contradiction.
Moreover, $\overline{N[v]}=\{y, u_{1}, v_{1}\}$ and $\overline{N[x]}=\{u, u_{1}, v_{1}\}$.
But  $\{y_{1}, y_{2}, y_{3}\}\neq\{y, u_{1}, v_{1}\}$ and
$\{y_{1}, y_{2}, y_{3}\}\neq\{u, u_{1}, v_{1}\}$. So $G-e'-S'$ is not $C3$.
Since $\{u_{1}, v_{1}\}$ is not an independent set of $G$, $G-e'-S'$ is not $C4$.

{\bf (2.2.1.2)} $|\overline{N[v]}\cap \overline{N[x]}|=1$.

Then $\overline{N[v]}\cap \overline{N[x]}=\{v_{1}\}$
and $u_{1}x\in E(G)$. Let $\{w\}=(N(u)\cap N(v))$$\setminus N(x)$.
By Proposition \ref{prop}, $G-e'-S'$ would be $C1$ or $C3$.

If $G-e'-S'$ is $C1$, then $S'=N(v)\cap N(x)\cup \{u_{1}\}$.
$G[\{u, v, w\}]$ and $G[\{x, y, v_{1}\}]$ are two odd components of $C1$.
But $yw, v_{1}w\in E(G)$. So $G-e'-S'$ is connected, a contradiction.

If $G-e'-S'$ is $C3$, then $y_{1}=v_{1}$.
Since $\overline{N[v]}=\{y, u_{1}, v_{1}\}$ and $u_{1}v_{1}\in E(G)$, $v$ does not
belong to the trivial odd component of $C3-\{a\}$. Moreover,
$\overline{N[x]}=\{u, v_{1}, w\}$ and $v_{1}w\in E(G)$. Then $x$ does not
belong to the trivial odd component of $C3-\{a\}$. So $G-e'-S'$ is not $C3$.

\smallskip
{\textbf{Subcase 2.2.2.} $v$ (resp. $u$) belongs to the trivial (resp. nontrivial)
odd component of $C3-\{a\}$.}

Then $\{y_{1}, y_{2}, y_{3}\}=\{u_{1}, y, v_{1}\}$.
Because $uu_{1}\in E(G)$ and $yu_{1}$ or $yv_{1}\in E(G)$,
so $y_{1}=v_{1}$. Then $u_{1}v_{1},$ $yv_{1}\notin E(G)$.
So $yu_{1}\in E(G)$ as $d_{G}(y)=n-4$ and $xv_{1}\in E(G)$ as $d_{G}(v_{1})\geq n-5$.

{\bf (2.2.2.1)} $|\overline{N[v]}\cap \overline{N[x]}|=0$.

Then $u_{1}x\in E(G)$. Let $\{w_{1}, w_{2}\}=(N(u)\cap N(v))$$\setminus N(x)$.
By Proposition \ref{prop}, $G-e'-S'$ would be $C1$. However,
the odd component of $C1$ containing $v$ must contain $w_{1}$ or $w_{2}$
and the odd component of $C1$ containing $x$ must contain $y$ or $v_{1}$.
But $yw_{1}$, $yw_{2}$, $v_{1}w_{1}$, $v_{1}w_{2}\in E(G)$.
In each case, $G-e'-S'$ is connected which is not $C1$.

{\bf (2.2.2.2)} $|\overline{N[v]}\cap \overline{N[x]}|=1$.

Then $u_{1}x\notin E(G)$ and $\overline{N[v]}\cap \overline{N[x]}=\{u_{1}\}$.
Let $\{w_{1}\}=$$(N(u)\cap N(v))$$\setminus N(x)$.
By Proposition \ref{prop}, $G-e'-S'$ would be $C1$ or $C3$.

If $G-e'-S'$ is $C1$, then the odd component of $C1$ containing $x$ must contain $y$
and the odd component of $C1$ containing $v$ must contain $u_{1}$ or $w_{1}$.
But $yu_{1}, yw_{1}\in E(G)$, a contradiction.

If $G-e'-S'$ is $C3$, then $y_{1}=u_{1}$.
Since $\overline{N[v]}=\{y, u_{1}, v_{1}\}$ and $yu_{1}\in E(G)$, $v$ does not
belong to the trivial odd component of $C3-\{a\}$.
Because $\overline{N[x]}=\{u, u_{1}, w_{1}\}$ and $uu_{1}\in E(G)$, $x$ does not
belong to the trivial odd component of $C3-\{a\}$. Then $G-e'-S'$ is not $C3$.

Thus, if $G-e-S$ is $C3$, then there is an edge $e'=vx$ such that $G-e'-S'$ is not one of
Configurations $C1$, $C2$, $C3$ and $C4$, which contradicts Lemma \ref{four}.

\smallskip
{\textbf{Subcase 2.3.} $G-e-S$ is $C4$.}

Then $\{y, v_{1}\}$ is an independent set of $G$. So $yv_{1}\notin E(G)$.

If $xv_{1}\notin E(G)$, then $y$ and $v_{1}$ are adjacent to every vertex in
$V(G)-\{u, v, x, y, v_{1}\}$.
We consider edge $e'=vx$ the same as Subcase~2.2.1.

If $xv_{1}\in E(G)$, then $|\overline{N[v]}\cap \overline{N[x]}|\leq 1$
and $y$ is adjacent to every vertex in $V(G)-\{u, v, y, v_{1}\}$.
We consider edge $e'=vx$ with similar discussion in Subcase~2.2.2.

\smallskip
Therefore, $G-e-S$ is not $C1$, $C3$ or $C4$ when $|\overline{N[u]}\cap \overline{N[v]}|=2$.

\smallskip
From the above discussion, for any $e=uv\in E(G)$,
there exists no $S\subseteq V(G)-\{u, v\}$ with $|S|=n-6$ such that
$G-e-S$ is one of Configurations $C1$, $C2$, $C3$ and $C4$,
which contradicts Lemma~\ref{four}. Then $G$ does not
contain a path $P$ with length three, in which each vertex has degree $n-4$.
\end{proof}

\end{document}